\documentclass[a4paper]{article}

\usepackage[utf8]{inputenc}
\usepackage[T1]{fontenc}
\usepackage{lmodern}
\usepackage{array}
\usepackage{graphicx}
\usepackage{amssymb,amsfonts,amsthm,amsmath,bbm,mathrsfs,aliascnt,mathtools}
\usepackage[english]{babel}
\usepackage[colorlinks]{hyperref}
\usepackage{tikz}

\usepackage{subfigure}

\theoremstyle{plain}
\newtheorem{theorem}{Theorem}[section]
\newtheorem{corollary}[theorem]{Corollary}
\newtheorem{lemma}[theorem]{Lemma}
\newtheorem{proposition}[theorem]{Proposition}

\newtheorem{conjecture}[theorem]{Conjecture}
\newtheorem{fact}[theorem]{Fact}

\theoremstyle{definition}
\newtheorem{definition}[theorem]{Definition}

\theoremstyle{remark}
\newtheorem{remark}[theorem]{Remark}

\newcommand{\N}{\mathbb{N}}
\newcommand{\Z}{\mathbb{Z}}
\newcommand{\R}{\mathbb{R}}

\newcommand{\ind}[1]{\mathbf{1}_{\left\{#1\right\}}}

\newcommand{\floor}[1]{{\left\lfloor #1 \right\rfloor}}

\numberwithin{equation}{section}

\DeclareMathOperator{\E}{\mathbb{E}}

\renewcommand{\P}{\mathbb{P}}

\newcommand{\calF}{\mathcal{F}}

\newcommand{\calU}{\mathcal{U}}

\newcommand{\calD}{\mathcal{D}}

\renewcommand{\bar}[1]{\overline{#1}}

\newcommand{\egaldistr}{\text{ }{\overset{(d)}{=}}\text{ }}

\newcommand{\T}{\mathbf{T}}
\renewcommand{\tilde}[1]{\widetilde{#1}}
\renewcommand{\hat}[1]{\widehat{#1}}

\renewcommand{\rho}{\varrho}
\renewcommand{\epsilon}{\varepsilon}

\title{Genealogy of the extremal process of the~branching~random~walk}
\author{Bastien Mallein\thanks{LAGA, Université Paris 13}}
\date{\today}

\begin{document}

\maketitle

\begin{abstract}
The extremal process of a branching random walk is the point measure recording the position of particles alive at time $n$, shifted around the expected position of the minimal position. Madaule~\cite{Mad15} proved that this point measure converges, as $n \to \infty$, toward a randomly shifted, decorated Poisson point process. In this article, we study the joint convergence of the extremal process together with its genealogical informations. This result is then used to describe the law of the decoration in the limiting process, as well as to study the supercritical Gibbs measures of the branching random walk.
\end{abstract}

\section{Introduction}

A branching random walk on $\R$ is a discrete time particle system on the real line, which can be defined as follows. It starts with a unique particle positioned at 0 at time 0. At each new time $n \in \N$, each particle alive at time $(n-1)$ dies, giving birth to children that are positioned according to i.i.d. versions of a random point measure, shifted by the position of their parent. We denote by $\T$ the genealogical tree of the branching random walk. For $u \in \T$, we write $V(u)$ for the position of the particle $u$ and $|u|$ for the time at which $u$ is alive. The branching random walk is the random marked tree $(\T,V)$.

We assume that the process is supercritical:
\begin{equation}
  \label{eqn:supercritical}
  \E\left( \#\{u \in \T : |u|=1\} \right) > 1.
\end{equation}
It is a well-known result from Galton-Watson processes theory that this assumption is equivalent to the fact that the surviving event $S = \{ \#\T = \infty\}$, in which the process never dies out, occurs with positive probability. Moreover, we assume the branching random walk to be in the boundary case:
\begin{equation}
  \label{eqn:boundary}
  \E\left( \sum_{|u|=1} e^{-V(u)} \right) = 1 \quad  \text{and} \quad  \E\left( \sum_{|u|=1} V(u)e^{-V(u)} \right) = 0,
\end{equation}
and that the reproduction law is non-lattice. This assumption guarantees that the minimal position $M_n = \min_{|u|=n} V(u)$ satisfies $\lim_{n \to \infty} \frac{M_n}{n} = 0$ a.s. (c.f. Biggins \cite{Big76}). Any branching random walk satisfying mild assumptions can be reduced to this case by an affine transformation (see e.g. the discussion in~\cite{BeG11}). We set
\[
  \forall n \geq 0, \quad W_n = \sum_{|u|=n} e^{-V(u)} \quad \text{and} \quad Z_n = \sum_{|u|=n} V(u)e^{-V(u)}.
\]
By \eqref{eqn:boundary} and the branching property of the branching random walk, the processes $(W_n)$ and $(Z_n)$ are martingales, which are called the critical martingale and derivative martingale of the branching random walk respectively.

We introduce the following additional integrability conditions:
\begin{align}
  &\sigma^2 := \E\left( \sum_{|u|=1} V(u)^2 e^{-V(u)} \right) \in (0,\infty)\label{eqn:variance}\\
  \text{and} \quad &\E\left(\sum_{|u|=1} e^{-V(u)} \log_+\left(\sum_{|u|=1} (1 + V(u)_+)e^{-V(u)}\right)^2 \right) < \infty\label{eqn:integrability},
\end{align}
where $x_+ = \max(x,0)$ and $\log_+(x) = \max(\log x,0)$. Under these assumptions, is is well-known (see \cite{Aid13,BiK04}) there exists a random variable $Z_\infty$, which is a.s. positive on the survival event $S$, such that
\begin{equation}
  \lim_{n \to \infty} Z_n = Z_\infty \quad \text{and} \quad \lim_{n \to \infty} W_n = 0 \quad \text{a.s.}
  \label{eqn:cvMartingales}
\end{equation}
Assumption \eqref{eqn:integrability} is a rephrasing of \cite[Equation (1.4)]{Aid13} (see Lemma~\ref{lem:equivalent} for the proof of the equivalence of these two integrability conditions). However, this version appears directly in our computations, cf Lemma~\ref{lem:entangled}.

Recall that $M_n = \min_{|u|=n} V(u)$ is the minimal position at time $n$ occupied by a particle. We set $m_n = \frac{3}{2} \log n$. Under the above integrability assumptions, Addario-Berry and Reed \cite{ABR09} observed that $(M_n - m_n)$ is tight, and Hu and Shi \cite{HuS09} proved this sequence has almost sure logarithmic size fluctuations. Finally, Aïdékon \cite{Aid13} obtained the convergence in law of $M_n - m_n$, and Chen \cite{Che15} proved the above integrability assumptions to be optimal for this convergence in law. We take interest in all particles that are at time $n$ in a $O(1)$ neighborhood of the minimal displacement $M_n$.

We introduce some notation on point measures, the Radon measures on $\R$ that takes values in $\Z_+ \cup \{ \infty\}$. Given a point measure $\rho$, we denote by $\mathcal{P}(\rho)$ the multiset of the atoms of the point measure $\rho$, that satisfy
\[
  \rho = \sum_{r \in \mathcal{P}(\rho)} \delta_r.
\]
For any $x \in \R$, we write $\theta_x \rho = \sum_{r \in \mathcal{P}(\rho)} \delta_{r + x}$ the shift of the measure $\rho$ by $x$. The space of point measures is endowed with the topology of the vague convergence, meaning that we write $\lim_{n \to \infty} \rho_n = \rho_\infty$ if $\lim_n \rho_n(f) = \rho_\infty(f)$ for any continuous function $f$ on $\R$ with compact support. As observed in \cite[Theorem A2.3]{Kal02}, the set of random point measures endowed with the topology of the vague convergence is a Polish space.

We use the extremal process of the branching random walk to record the positions of particles close to the maximal displacement at time $n$, defined by
\begin{equation}
  \label{eqn:defineExtremalProcess}
  \gamma_n = \sum_{|u|=n} \delta_{V(u) - m_n}.
\end{equation}
Madaule \cite{Mad15} proved the convergence in law of the extremal process toward a decorated Poisson point process with exponential intensity, or more precisely the following result.
\begin{fact}[Theorem 1.1 in \cite{Mad15}]
\label{fct:Madaule}
We assume \eqref{eqn:supercritical}, \eqref{eqn:boundary}, \eqref{eqn:variance} and \eqref{eqn:integrability}. There exist $c_*>0$ and a point measure $D$ satisfying $\min D = 0$ a.s. such that
\[
  \lim_{n \to \infty} \left(\gamma_n,Z_n\right) = \left( \gamma_\infty, Z_\infty \right) \quad \text{in law on $S$,}
\]
for the topology of the vague convergence, where $(\xi_n)$ are the atoms of a Poisson point process with intensity $c_* e^x dx$, $(D_n, n \geq 1)$ are i.i.d. copies of $D$ and
\begin{equation}
  \label{eqn:rhoInfinity}
  \gamma_\infty = \sum_{n=1}^{\infty} \theta_{\xi_n- \log Z_\infty} D_n.
\end{equation}
We denote by $\mathcal{D}$ the law of the point measure $D$.
\end{fact}
The point measure $\gamma_\infty$ is called a shifted decorated Poisson point process with shift $-\log Z_\infty$ and decoration law $\calD$ (or SDPPP($c_*e^{x}dx$,$-\log Z_\infty$, $\calD$) for short). These point measures have been studied in particular by Subag and Zeitouni \cite{SuZ15}. The proof of Fact~\ref{fct:Madaule} gives little information on the law $\calD$ of the point measure $D$ used for the decoration of $\gamma_\infty$. Indeed, the convergence of $(\gamma_n)$ is obtained through the study of its Laplace transform, and the law of the limiting point measure $\gamma_\infty$ is identified using its superposability property\footnote{Cf. Maillard \cite{Mai13} for the characterization of point measures occurring as the limits of the extremal processes of branching random walks.}.

A result similar to Fact~\ref{fct:Madaule} was previously obtained for the branching Brownian motion independently by Arguin, Bovier and Kistler \cite{ABK13}, and Aïdékon, Berestycki, Brunet and Shi \cite{ABBS13}. In this model as well, the extremal process converges toward a decorated Poisson point process. However, the decoration law is explicitly described in both these articles. In \cite{ABBS13}, the point measure $D$ corresponds to positions of the close relatives of the particle realizing the minimal displacement. In \cite{ABK13}, it is described as the extremal process of the branching random walk conditioned on having an unusually small minimal displacement.

In this article, we observe that using the branching property as well as an enriched version of the extremal process, Fact~\ref{fct:Madaule} immediately implies a stronger version of itself. More precisely, thanks to a careful encoding of the genealogy of the branching random walk, which is presented in Section~\ref{sec:notation}, we can prove the joint convergence in law of the extremal process with some genealogical informations in Section~\ref{sec:cv}. This convergence yields the observation that in the point process $\gamma_\infty$, the Poisson point process correspond to leaders realizing independently a small displacement, while each decoration comes from the family of the close relatives to a leader. This is reminiscent of the result obtained in \cite{ABK12} in the context of branching Brownian motion.

Similar results of convergence of enriched extremal processes have been recently obtained by Biskup and Louidor \cite{BiL} for the 2 dimensional Gaussian free field, and by Bovier and Hartung \cite{BoH} for the branching Brownian motion. Cortines, Hartung and Louidor \cite{CHL17} obtained refined results on the law of the decoration of the branching Brownian motion using among other things enrichment of the extremal process techniques. This method thus seems promising. For example, it might be used to proved simultaneous convergence in law of the rescaled trajectories of extremal particles toward Brownian excursions, as conjectured in \cite{CMM}, i.e. that for all $\beta > 1$
\begin{equation}
  \label{eqn:conjCMM}
  \lim_{n \to \infty} \frac{1}{\sum_{|u|=n} e^{-\beta V(u)}} \sum_{|u|=n} e^{-\beta V(u)}\delta_{H_n(u)} = \sum_{k \in \N} p_k \delta_{\mathbf{e}_k} \quad \text{ in law},
\end{equation}
where $H_n(u) : t \in [0,1] \mapsto n^{-1/2}V(u_{\floor{nt}})$, $(p_k, k \geq 1)$ is a Poisson-Dirichlet distribution with parameters ($\beta^{-1},0$), and $(\mathbf{e}_k, k \geq 1)$ are i.i.d. standard Brownian excursions.

We use here the convergence of the extremal process with genealogical informations to obtain simple proofs for a few additional results. We study the weak convergence of the so-called supercritical Gibbs measure of the branching random walk, as obtained in \cite{BRV12}. We also prove a conjecture of Derrida and Spohn on the asymptotic behavior of the so-called overlap of the branching random walk. More precisely, conditionally on the branching random walk $(\T,V)$ we select two particles $u^{(n)}$, $v^{(n)}$ at the $n$th generation with probability proportional to $e^{-\beta (V(u) + V(v))}$, and denote by $\omega_{n,\beta}$ the law of the age of their most recent common ancestors, rescaled by a factor $n$. We prove that $(\omega_{n,\beta})$ converges, as $n \to \infty$ toward the probability measure $(1-\pi_\beta)\delta_0 + \pi_\beta \delta_1$, where $\pi_\beta$ is a random variable whose law depend on $\beta$.

As an other application of the convergence of the enriched extremal process, we finally obtain a description of the law $\calD$ of the decoration of this process as the limit of the position of close relatives of the minimal displacement at time $n$. This result mimics the one proved in \cite{ABBS13} for the decoration of the branching Brownian motion. We expect a result similar to \cite{ABK13} would also holds, i.e. that the law $\calD$ could be obtained as the limit in distribution of the extremal process conditioned on having a very small minimum.

\paragraph*{Outline.}
In the next section, we precise the encoding of the branching random walk, and use it to define the so-called critical measure: a measure on the boundary of the tree $\T$ of the branching random walk, whose distribution is related to the derivative martingale. We prove in Section~\ref{sec:cv} the convergence of the enriched extremal process. In Section~\ref{sec:applications}, we use this enriched convergence to prove the weak convergence of the supercritical Gibbs measure and the Derrida--Spohn conjecture. The expression of the law of the decoration $\calD$ as the position of close relatives of the minimal displacement at time $n$ is obtained in Section~\ref{sec:decoration}.

\section{The critical measure of the branching random walk}
\label{sec:notation}

In this section, we first introduce the so-called Ulam-Harris notation for trees, that is used for a precise definition of the the branching random walk. In a second time we define the so-called critical measure of the branching random walk and study some of its properties. This measure is defined on the boundary of the tree of the branching random walk, and its distribution is related to the derivative martingale.

\subsection{Construction of the branching random walk}
\label{subsec:ulam}

We introduce the sets
\[
  \calU = \bigcup_{n \geq 0} \N^n ,\quad \partial \calU = \N^\N \quad \text{and} \quad \bar{\calU} = \calU \cup \partial \calU,
\]
with the convention $\N^0 = \{ \emptyset \}$. In the Ulam-Harris notation, a (plane, rooted) tree is constructed as a subset of $\calU$, each element $u \in \calU$ representing a potential individual.

Let $u \in \bar{\calU}$, which is a finite or infinite sequence of integers. We denote by $|u|$ the length of the sequence $u$ and, for $k \leq |u|$ by $u_k$ the sequence consisting of the $k$ first values of $u$. If $u \in \mathcal{U} \backslash \{\emptyset\}$, we write $\pi u = u_{|u|-1}$ the sequence obtained by erasing the last element. For $u \in \calU$ and $v \in \bar{\calU}$, we denote by $u.v$ the concatenation of the sequences. For $u,v \in \bar{\calU}$, we write $u \leq v$ if $v_{|u|}=u$, which define a partial order on $\bar{\calU}$. We then define $|u \wedge v| = \max\{ k \in \N : u_k = v_k\}$ and $u\wedge v = u_{|u \wedge v|} = v_{|u \wedge v|}$.

The genealogical tree $\T$ of the branching random walk is encoded as a subset of $\mathcal{U}$ in the following way. The root is encoded by the empty sequence $\emptyset$, while $u = (u(1), \ldots u(n)) \in \mathcal{U}$ represents the $u(n)$th child of the $u(n-1)$th child of the ... of the $u(1)$th child of the root. With this encoding, $\pi u$ is the parent of $u$, $|u|$ the generation to which $u$ belongs, $u_k$ the ancestor of $u$ at generation $k$. We write $u<v$ if $u$ is an ancestor of $v$, and $u \wedge v$ is the most recent common ancestor of $u$ and $v$.

The family of positions $(V(u), u \in \T)$ is then a random map from $\T$ to $\R$, which can be extended as a random map $\calU \to \R \cup\{-\infty\}$, by setting $V(u) = -\infty$ for $u \in \calU \backslash \T$. We then call $V : \calU \to \R \cup\{-\infty\}$ the branching random walk, which can be constructed as follows. Let $\{(\ell^u_j, j \in \N), u \in \calU\}$ be a family of i.i.d. random variables in $(\R \cup\{-\infty\})^\N$, we set
\[
  V(u) = \sum_{j=1}^{|u|} \ell^{u_{j-1}}_{u(j)},
\]
with the convention $-\infty + x = x - \infty = -\infty$ for all $x \in \R \cup \{-\infty\}$. The law of $(\ell^\emptyset_j, j \in \N)$ is called the reproduction law of the branching random walk $V$. Note that one can recover $\T$ from $V$ as $\{ u \in \calU : V(u) > -\infty\}$.

\subsection{A topology on the set of leaves of infinite trees}

With the above notation, the set $\partial \calU$ represents the set of possible leaves in the tree $\T$, infinite non-backtracking paths starting from the root. The critical measure of the branching random walk that we now describe is constructed as a Radon measure on the set of leaves. In this section, we introduce a topology on $\bar{\calU}$ that makes it a compact space, and observe that finite measures on that space are identified with flows on the tree $\calU$.

We embed $\bar{\calU}$ in $[0,1]$, observing that the application
\[\Psi : u \in \bar{\calU} \longmapsto 2\sum_{j=1}^{|u|} 3^{-\sum_{i=1}^j u(i)}\]
is a bijection between $\bar{\calU}$ and the Cantor ternary set $K$, depicted in Figure~\ref{fig:Cantor}.

\begin{figure}
\label{fig:Cantor}
\centering
\begin{tikzpicture}

\draw [line width=0.3cm] (0.0,0) -- (0.099,0);
\draw [line width=0.3cm] (0.198,0) -- (0.296,0);
\draw [line width=0.3cm] (0.593,0) -- (0.691,0);
\draw [line width=0.3cm] (0.79,0) -- (0.889,0);
\draw [line width=0.3cm] (1.778,0) -- (1.877,0);
\draw [line width=0.3cm] (1.975,0) -- (2.074,0);
\draw [line width=0.3cm] (2.37,0) -- (2.469,0);
\draw [line width=0.3cm] (2.568,0) -- (2.667,0);
\draw [line width=0.3cm] (5.333,0) -- (5.432,0);
\draw [line width=0.3cm] (5.531,0) -- (5.63,0);
\draw [line width=0.3cm] (5.926,0) -- (6.025,0);
\draw [line width=0.3cm] (6.123,0) -- (6.222,0);

\draw [color=red, line width=0.34cm] (7.085,0) -- (7.23,0);
\draw [color=red, line width=0.34cm] (7.279,0) -- (7.427,0);
\draw [color=red, line width=0.34cm] (7.674,0) -- (7.822,0);
\draw [color=red, line width=0.34cm] (7.871,0) -- (8.02,0);

\draw [line width=0.3cm] (7.111,0) -- (7.21,0);
\draw [line width=0.3cm] (7.309,0) -- (7.407,0);
\draw [line width=0.3cm] (7.704,0) -- (7.802,0);
\draw [line width=0.3cm] (7.901,0) -- (8.0,0);

\draw [color = red] (8.7,0) node {$B(1,1)$};

\draw [->,thick,color=blue] (0,1.8) -- (0,0.25);
\draw [color=blue] (-0.0,1.8) node[above] {$\Psi(\emptyset)$};
\draw [->,thick,color=blue] (5.333,1.2) node[above] {$\Psi(1)$} -- (5.333,0.25);
\draw [->,thick,color=blue] (1.778,1.2) node[above] {$\Psi(2)$} -- (1.778,0.25);
\draw [->,thick,color=blue] (0.593,1.2)  -- (0.593,0.25);
\draw [color=blue] (0.593,1.2) node[above] {$\Psi(3)$};
\draw [->,thick,color=blue] (7.111,0.7) node[above] {$\Psi(1,1)$} -- (7.111,0.25);
\draw [->,thick,color=blue] (5.926,0.7) node[above] {$\Psi(1,2)$} -- (5.926,0.25);

\draw [line width=0.3cm] (0.0,-0.5) -- (0.296,-0.5);
\draw [line width=0.3cm] (0.593,-0.5) -- (0.889,-0.5);
\draw [line width=0.3cm] (1.778,-0.5) -- (2.074,-0.5);
\draw [line width=0.3cm] (2.37,-0.5) -- (2.667,-0.5);
\draw [line width=0.3cm] (5.333,-0.5) --(5.63,-0.5);
\draw [line width=0.3cm] (5.926,-0.5) -- (6.222,-0.5);
\draw [line width=0.3cm] (7.111,-0.5) -- (7.407,-0.5);
\draw [line width=0.3cm] (7.704,-0.5)-- (8.0,-0.5);

\draw [line width=0.3cm] (0.0,-1)  -- (0.889,-1);
\draw [line width=0.3cm] (1.778,-1) -- (2.667,-1);
\draw [line width=0.3cm] (5.333,-1)-- (6.222,-1);
\draw [line width=0.3cm] (7.111,-1) -- (8.0,-1);

\draw [line width=0.3cm] (0.0,-1.5)  -- (2.667,-1.5);
\draw [line width=0.3cm] (5.333,-1.5)-- (8.0,-1.5);

\draw [line width=0.3cm] (0.0,-2)  -- (8.0,-2);

\end{tikzpicture}
\caption{Mapping between $\bar{\mathcal{U}}$ and the Cantor ternary set}
\end{figure}
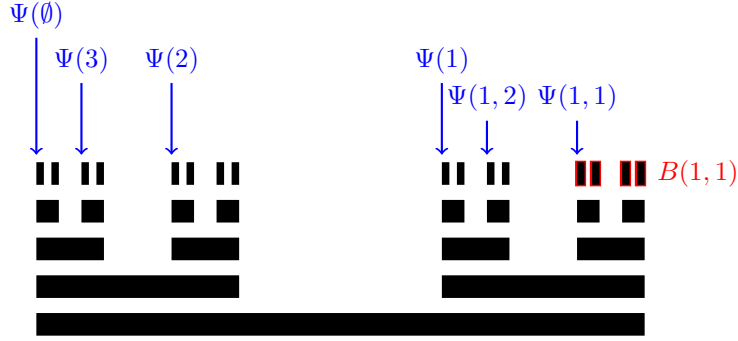

Using this bijection, we define a distance on $\bar{\calU}$ by
\[
  \forall u, v \in \bar{\calU}, \quad d(u,v) = 2^{1-\min\{ n \in \N : 3^n |\Psi(u)-\Psi(v)| \geq 1\}},
\]
with the convention that $\min \emptyset = \infty$ and $2^{-\infty} = 0$. Note this distance can be rewritten as
\[
  d(u,v) = 2^{1- \min(u(|u \wedge v|+1), v(u \wedge v|+1))- \sum_{j=1}^{|u \wedge v|} u(j)},
\]
with the convention that if $|u|=n$, then $u(n+1) = 0$. This distance measures to which depth of construction of the Cantor set one should go before the images of $u$ and $v$ by $\Psi$ are in distinct blocs. It is thus straightforward that $\bar{\calU}$ is a compact ultrametric space when endowed by this distance. Informally, the topology of $(\bar{\calU},d)$ can be described as the topology of pointwise convergence for infinite sequences of integers, with the addition that $\lim_{n \to \infty} u.n = u$, or in other words, identification between the sequence $(u(1),\ldots u(n), \infty, v(1),\ldots)$ with $u$, where $u \in \N^n$ and $v \in (\N \cup \infty)^\N$.

Note that $\calU$ is a dense countable subset of $\bar{\calU}$ for this topology. For any $u \in \mathcal{U}$, we denote by
\[
  B(u) =  \left\{ v \in \bar{\calU} : u \wedge v =u \right\} = \left\{ u.w, w \in \bar{\calU} \right\}/
\]
Observe that $\{B(u), u \in \calU\}$ is a family of open and close balls of $\bar{\calU}$ for the distance $d$. We also set, for $u \in \calU$ and $n \in \N$
\begin{equation}
  \label{def:countableBase}
  C(u,n) = B(u) \backslash \left( \cup_{j=1}^{n-1} B(u.j)\right) = \{u\} \cup \bigcup_{j = n}^{\infty} B(u.j).
\end{equation}

\begin{lemma}
\label{lem:countableBase}
The family $\{C(u,j), u \in \mathcal{U}, j \in \N\}$ forms a countable base of open sets for $(\bar{\calU},d)$.
\end{lemma}

\begin{proof}
Let $\mathcal{O}$ be an open subset of $\bar{\calU}$. We define
\begin{equation}
  \Gamma(\mathcal{O}) = \left\{ u \in \mathcal{U} : C(\pi u, u(|u|)) \not \subset \mathcal{O}, \exists j \in \N : C(u,j) \subset \mathcal{O}\right\}
\end{equation}
as well as $j_u = \inf\{ j \in \N : C(u,j) \subset \mathcal{O} \}$ for $u \in \Gamma(\mathcal{O})$. We observe that
\begin{equation}
  \label{eqn:countableBase}
  \mathcal{O} = \bigcup_{u \in \Gamma(\mathcal{O})} C(u,j_u),
\end{equation}
and that the union is of pairwise disjoint elements.

Indeed, for any $v \in \mathcal{O}$, there exists $n \in \N$ such that $B(v_n) \subset \mathcal{O}$. We denote by $n_0 = \inf\{ n \in \N : B(v_n) \subset \mathcal{O}\}$. Then either $v_{n_0} \in \Gamma(\mathcal{O})$ and $j_{v_0}=1$, or $v_{n_0-1} \in \Gamma(\mathcal{O})$ and $j_{v_{n_0-1}} \leq v(n_0)$. Reciprocally, it follows from definition that $C(u,j_u) \subset \mathcal{O}$ for all $u \in \Gamma(\mathcal{O})$.

We now observe that $u \in C(u,n)$ for all $n \in \N$. Hence, as $C(u,n) \subset B(u)$, if $C(u,n) \cap C(v,m) \neq \emptyset$, then either $u$ is an ancestor of $v$, or $v$ is an ancestor of $u$, or $u=v$. Moreover, if $u$ is an ancestor of $v$, then $v \in C(u,n)$.

We assume there exists $u \neq v \in \Gamma(\mathcal{O})$ such that $C(u,j_u) \cap C(v,j_v) \neq \emptyset$. We can assume without loss of generality that $u$ is an ancestor of $v$, hence $v = u. l . w$, with $l \in \N$ and $w \in \calU$. As $v \in C(u,j_u) $, we have $l \geq j_u$. Hence, by definition, we have $B(u.l) \subset \mathcal{O}$, which is in contradiction with the fact that $v \in \Gamma(\mathcal{O})$.
\end{proof}

Note that \eqref{eqn:countableBase} can be rewritten as the disjoint union of elements belonging to the families $\{B(u), u \in \mathcal{U}\}$ and $\{\{u\}, u \in \mathcal{U}\}$.

In the rest of the section, we study finite measures on the space $(\bar{\calU},d)$. We first identify them with pseudo-flows on $\calU$, which we now define.
\begin{definition}
A function $f : \calU \to \R_+$ is called a pseudo-flow on $\calU$ if it satisfies
\begin{equation}
  \label{eqn:pseudoflow}
  \forall u \in \calU, \quad f(u) \geq \sum_{j \in \N} f(u.j).
\end{equation}

A function $f : \calU \to \R_+$ is called a flow on $\calU$ if it satisfies
\begin{equation}
  \forall u \in \calU, \quad f(u) = \sum_{j \in \N} f(u.j).
\end{equation}
\end{definition}

To each Radon measure $\mu$ on $(\bar{\calU},d)$ we can associate a pseudo-flow on $\calU$ defined as
\[
  \forall u \in \calU, \quad f_\mu(u) = \mu(B(u)).
\]
Note that the function $f_\mu$ is a flow if and only if $\mu(\calU) = 0$. We now observe that $\mu \mapsto f_\mu$ realizes a bijection between the Radon measures and the pseudo-flows.
\begin{proposition}
\label{prop:flows}
For each pseudo-flow $f$ on $\calU$, there exists a unique finite measure $\mu$ such that $f=f_\mu$.

A measure $\mu$ is atomless if and only if $f_\mu$ is a flow and
\[
  \lim_{n \to \infty} \max_{|u|=n} f_\mu(u) = 0.
\]
\end{proposition}

\begin{proof}
Let $\mu$ and $\nu$ be two finite measures such that $f_\mu = f_\nu$. By definition, this indicates that
\[
  \forall u \in \mathcal{U}, \quad \mu(B(u)) = \nu(B(u)).
\]
Therefore, by sigma-additivity, one successively deduces that
\[
  \forall u \in \mathcal{U}, n \in \N, \quad \mu(C(u,n)) = \sum_{j=n}^\infty \mu(B(u.j)) = \sum_{j=n}^\infty \nu(B(u.j)) = \nu(C(u,n)),
\]
and, thanks to \eqref{eqn:countableBase}, that $\mu(\mathcal{O}) = \nu(\mathcal{O})$ for all open subset of $\bar{\calU}$. By monotone classes theorem, we deduce that $\mu = \nu$.

Let $f$ be a pseudo-flow on $\calU$, we now construct a measure on $\bar{\calU}$ associated to that pseudo-flow. We first observe that if $f(\emptyset) = 0$, then $f(u) = 0$ for all $u \in \calU$, hence the null measure is associated to that flow.

We now assume that $f(\emptyset) \neq 0$. Up to replacing $f$ by $f/f(\emptyset)$ we can assume without loss of generality that $f(\emptyset) = 1$. We observe that for all $n \in \N$, we can define the law $\mu_n$ on $\{ u \in \mathcal{U} : |u| \leq n\}$ by
\begin{align*}
  \forall u \in \calU : |u| < n, \quad \mu_n(u) = f(u) - \sum_{j =1}^\infty f(u.j)\\
  \forall u \in \calU : |u|=n, \quad \mu_n(u) = f(u). 
\end{align*}
Observe that thanks to the pseudo-flow property, the family of probability distributions $(\mu_n, n \geq 1)$ is consistent under the family of projections $\pi_n : u \mapsto u_n$. Hence, thanks to Kolmogorov extension theorem, there exists a probability measure $\mu$ on $\bar{\calU}$ such that its image measure by $\pi_n$ is $\mu_n$ for all $n \in \N$. Hence, one has straightforwardly $f_\mu = f$.

The second point is a straightforward consequence as if $\mu$ has an atom of mass $x$ at point $u \in \bar{\calU}$, then
\[
  f_\mu(u) = x + \sum_{j=1}^\infty f_\mu(u.j),
\]
if $u \in \calU$, by sigma-additivity, and
\[
  \lim_{n \to \infty} f_\mu(u_n) = x
\]
if $u \in \partial \calU$, by dominated convergence.
\end{proof}

\subsection{The critical measure of the branching random walk}
\label{subsec:critical}

Let $V$ be a branching random walk satisfying \eqref{eqn:supercritical}, \eqref{eqn:boundary}, \eqref{eqn:variance} and \eqref{eqn:integrability}. We recall that we set $\T = \{ u \in \calU : V(u) \neq -\infty\}$. For any $n \in \N$, we denote by $\T(n) = \{ u \in \T : |u|=n\}$ the set of individuals alive at generation $n$. We introduce the filtration $(\calF_n)$, defined by
\[
  \calF_n = \sigma\left( V(u), u \in \mathcal{U} : |u| \leq n \right).
\]
Note that by definition of the branching random walk, if we set, for all $u \in \T$
\begin{equation}
  \label{eqn:subBrw}
  V^u : v \in \calU \mapsto V(u.v) - V(u)
\end{equation}
the branching random walk issued from particle $u$, then by definition, for all $n \in \N$, $\{V^u, u \in \T(n)\}$ is a family of i.i.d. branching random walks, with same law as $V$, that are independent from $\calF_n$. This fact is often called the branching property of the branching random walk.

We denote the boundary of the branching random walk by
\[
  \partial \T = \left\{ u \in \partial \calU : \forall n \in \N, u_n \in \T\right\}.
\]
An element of $\partial \T$ represent a spine of the tree: a semi-infinite path starting at and going away from the root in the tree $\T$.

The critical measure of the branching random walk has been introduced by the physicists Derrida and Spohn in \cite{DeS88}. Its existence is a consequence the precise study of the derivative martingale in \cite{AiS14}. This measure has been the subject of multiple studies \cite{BKNSW,Bur09,BDK16+}.

To define the critical measure, for any $u \in \T$, we set
\[
  Z^u_n = \sum_{|v|=n, v > u} (V(v)-V(u))e^{V(u)-V(v)} \quad \text{and} \quad Z^u_\infty = \liminf_{n \to \infty} Z^u_n.
\]
Thanks to the branching property, we observe $\left(Z^u_\infty, u \in \T(k)\right)$ are i.i.d. copies of $Z_\infty$, which are independent of $\calF_k$. Moreover, for any $k \leq n$, we have
\[
  Z_n = \sum_{|u|=k} e^{-V(u)} Z^u_{n} + \sum_{|u|=k} V(u)e^{-V(u)} \sum_{|v|=n,v>u} e^{V(u)-V(v)}.
\]
Letting $n \to \infty$ and using \eqref{eqn:cvMartingales}, we deduce $Z_\infty = \sum_{|u|=k} e^{-V(u)} Z^u_\infty$ a.s. More generally, almost surely, for all $u \in \T$ we have
\begin{equation}
  \label{eqn:flowDerivative}
  e^{-V(u)} Z_\infty^u = \sum_{j \in \N} e^{-V(u.j)}Z_\infty^{u.j}.
\end{equation}

In other words, the function
\[
  f^* : u \in \calU \mapsto \begin{cases} e^{-V(u)} Z_\infty^u & \quad \text{ if } u \in \T\\ 0 & \quad \text{otherwise,}\end{cases}
\]
is a.s. a flow on $\calU$. The critical measure of the branching random walk is the unique measure $\nu$ on $\bar{\calU}$ associated to the flow $f^*$, i.e.
\begin{equation}
  \label{eqn:defineCriticalMeasure}
  \forall u \in \calU, \quad \nu\left( B(u) \right) = \ind{u \in \T} e^{-V(u)} Z_\infty^u \quad \text{a.s.}
\end{equation}
Existence and uniqueness of $\nu$ are proved in Proposition~\ref{prop:flows}. Moreover, note that as $Z_\infty >0$ a.s. on the survival event $S$ of the branching random walk, the support of $\nu$ is a.s. the adherence of the boundary of the tree $\partial \T$, for the distance $d$.

\begin{remark}
\label{rem:cvCritical}
Note that the following convergence holds
\[
  \nu = \lim_{n \to \infty} \sum_{|u|=n} V(u) e^{-V(u)} \delta_{u} \quad \text{a.s.}
\]
for the topology of weak convergence of measures on $\bar{\calU}$, as for all $v \in \mathcal{U}$, we have
$ \displaystyle \lim_{n \to \infty} \sum_{|u|=n} \ind{u \in B(v)}  V(u) e^{-V(u)} = \nu(B(v))$ a.s. We conclude using the Portmanteau theorem.
\end{remark}

We end this section with a short proof that $\nu$ is non-atomic. We first note that that as $f^*$ is a (proper) flow, we have immediately $\nu(\{u\}) = 0$ a.s for all $u \in \calU$, therefore $\nu(\calU)=0$ a.s. As a result, the fact that $\nu$ is non-atomic is a consequence of the following lemma.
\begin{lemma}
\label{lem:nonAtomic}
Under assumptions \eqref{eqn:supercritical}, \eqref{eqn:boundary}, \eqref{eqn:variance} and \eqref{eqn:integrability}, we have
\[
  \lim_{n \to \infty} \max_{|u|=n} \nu(B(u)) = 0 \quad \text{a.s.}
\]
\end{lemma}

\begin{proof}
We first recall the precise estimate on the tail of $Z_\infty$ obtained by Madaule \cite{Mad16}: there exists $c_1>0$ such that for any $x \geq 0$, $\P(Z_\infty \geq x) \leq \frac{c_1}{x}$.

We now observe that for any $\epsilon>0$ and $n\in \N$, we have
\begin{align*}
  \P\left( \max_{|u|=n} \nu(B(u)) \geq \epsilon \middle| \calF_n \right) &\leq \sum_{|u|=n} \P\left( e^{-V(u)} Z^u_\infty \geq \epsilon \middle| \calF_n \right)\\
  &\leq \frac{c_1}{\epsilon} \sum_{|u|=n} e^{-V(u)} \text{ a.s,}
\end{align*}
thus $\lim_{n \to \infty} \P\left( \max_{|u|=n} \nu(B(u)) \geq \epsilon \middle| \calF_n \right) = 0$ a.s. by \eqref{eqn:cvMartingales}. As the sequence $(\max_{|u|=n} \nu(B(u)), n \geq 0)$ is non-increasing in $n$, we conclude that
\[\lim_{n \to \infty} \max_{|u|=n} B(u) = 0 \quad \text{ a.s.}  \tag*{\qedhere}\]
\end{proof}

\section{Convergence in law of the extremal process with genealogical informations}
\label{sec:cv}

Using the notation of the previous section, we can now state the main result of this paper, namely the convergence of the point measure
\begin{equation}
  \label{eqn:genalogicalExtremalProcess}
  \mu_n = \sum_{|u|=n} \delta_{u,V(u)-m_n},
\end{equation}
on $\bar{\calU} \times \R$. The sketch of proof is the following: we first define a candidate for the limiting measure, then observe that Fact~\ref{fct:Madaule} implies that $\mu_n$ converges toward this well-chosen limiting measure.

Independently from the branching random walk $V$, let $(\xi_n, n \geq 1)$ be the atoms of a Poisson point process with intensity $c_* e^{x}dx$, $(u^{(n)}, n \geq 1)$ be i.i.d. random variables with law $\bar{\nu}$ and $(D_n, n \geq 1)$ be i.i.d. point measures with law~$\calD$, with $c_*$ and $\calD$ defined in Fact~\ref{fct:Madaule}. For any $n \in \N$, we set
\begin{equation}
  \label{eqn:limitlExtremalProcess}
  \mu_\infty = \sum_{n=1}^{\infty} \sum_{d \in \mathcal{P}(D_n)} \delta_{u^{(n)},\xi_n + d-\log Z_\infty}.
\end{equation}
By classical properties of Poisson point processes, $( u^{(n)},\xi_n-\log Z_\infty)$ are the atoms of a Poisson point process with intensity $c_*\nu\otimes e^x dx$ on $\bar{\calU}\times\R$. Hence $\mu_\infty$ can alternatively be described as a Poisson point process with intensity $c_*\nu \otimes e^x dx$, with an i.i.d. decoration on the second coordinate. The main result of the article is the following convergence.
\begin{theorem}
\label{thm:main}
Assuming \eqref{eqn:supercritical}, \eqref{eqn:boundary}, \eqref{eqn:variance} and \eqref{eqn:integrability}, we have 
\[\lim_{n \to \infty} (\mu_n,Z_n) = (\mu_\infty,Z_\infty) \quad \text{ in law on $S$},\]
for the topology of the vague convergence.
\end{theorem}

\begin{remark}
The genealogical informations encoded in $\mu_n$ only concern the local behavior in a neighborhood of the root of the process. Informally, we say that two individuals do not belong to the same family if the age of their most recent common ancestor if $O(1)$. However, we know that with high probability, for individuals close to the minimal displacement at time $n$, the age of their most recent common ancestor is either $O(1)$ or $n-O(1)$ with high probability (see e.g. \eqref{eqn:genealogy}). But to obtain informations on the genealogy within the group of the followers, different quantities should be considered, such as the branching random walk seen from the local leader, for the topology of local convergence.
\end{remark}

The convergence in Theorem~\ref{thm:main} can be interpreted as follows. We can decompose the extremal process at time $n$, near position $m_n$ into families of individuals whose common ancestor was alive at generation $n-O(1)$. In each of these families, there is a leader, a particle whose position is the smallest within the family. The point process of the leaders converge toward a Poisson point process with exponential intensity, and the relative positions of their relatives converge toward i.i.d. copies of a point process of law $\calD$. The fact that $\nu$ has no atom proves that with high probability, the most recent common ancestor between two individuals of two distinct families was alive at time $O(1)$.

This convergences gives some informations on the genealogical relationships for particles close to the smallest position at time $n$. For example, in the non-lattice case, if two particles $u$ and $v$ are at position $M_n$ at time $n$, they are close relatives with high probability. Note that this result would not hold in the lattice case, as observed by Pain \cite[Footnote 3]{Pai17+}.

\begin{proof}[Proof of Theorem~\ref{thm:main}]
This is a direct consequence of Madaule's convergence in law for the extremal process of the branching random walk with its genealogy. For any $v \in \T$, we denote by
\[\mu^v_\infty(.) = \int_{B(v)} \mu_\infty(du,.) = \sum_{k=1}^{\infty} \ind{u^{(k)} > v} \theta_{\xi_k-\log Z_\infty} D_k,\]
For every $k \in \N$, conditionally on $\calF_k$ and $(Z^v_\infty, |v|=k)$, we observe that $(\theta_{-V(v)} \mu^v_\infty, v \in \T(k))$ are independent SDPPP($c_*e^{x}dx,-\log Z^v_\infty,\calD$).

In other words, $\mu^v_\infty$ has the same law as the limit of the extremal process of the branching random walk $V^v$ issued from particle $v$, defined in \eqref{eqn:subBrw}. Thus, by Fact~\ref{fct:Madaule}, conditionally on $\calF_k$, for any $v \in \T(k)$, we have
\[
  \lim_{n \to \infty} \left(\sum_{|u|=n,u>v} \delta_{V(u)-V(v)-m_n}, Z^v_n \right) = \left(\theta_{-V(v)}\mu^v_\infty, Z^v_\infty\right) \text{ in law on $S$.}
\]

We denote by $f$ a continuous non-negative function on $\R$ with compact support and $k \in \N$. By the branching property, conditionally on $\calF_k$, the subtrees of the branching random walk rooted at points $v \in \T(k)$ behave as independent branching random walk. Therefore
\begin{align*}
  &\lim_{n \to \infty} \left(\left(\mu_n(\mathbf{1}_{B(v)}f), v \in \T(k)\right),Z_n \right)\\
  = &\lim_{n \to \infty} \left(\left(\sum_{|u|=n, u > v} f(V(u)-m_n),v \in \T(k)\right), \sum_{|v|=k} e^{-V(v)} Z_n^v \right)\\
  = &\left(\left(\mu^v_\infty(f), v \in \T(k)\right), Z_\infty\right) \quad \text{ in law on $S$}.
\end{align*}
By \cite[Theorem 14.16]{Kal02}, we conclude that $\lim_{n \to \infty} (\mu_n,Z_n) = (\mu_\infty,Z_\infty)$ in law on the survival event $S$.
\end{proof}

Using \cite[Theorem 10]{SuZ15}, we observe that writing $\hat{\mu}^x$ for a point measure with distribution $\theta_{-\min \mu_\infty} \mu_\infty$ conditionally on $\{\min \mu_\infty < -x\}$, we have
\[
 \lim_{x \to \infty} \hat{\mu}^x = D_1 \quad \text{ in law}.
\]
This result can be seen as a (weaker form of the) characterization of \cite{ABK13} of the law $\calD$. We provide an alternative characterization of this law in Section~\ref{sec:decoration}.

A straightforward consequence of Theorem~\ref{thm:main} is the convergence for the extremal process seen from the smallest position.
\begin{corollary}
\label{cor:main}
Under the same assumptions as Theorem~\ref{thm:main}, we set $\mathbf{e}$ a standard exponential random variable, $\zeta_1=0$ and $(\zeta_n, n \geq 2)$ a Poisson point process with intensity $\mathbf{e}e^{x}\ind{x>0}dx$. We have 
\[
  \lim_{n \to \infty} \left(M_n - m_n, \sum_{|u|=n} \delta_{u,V(u)-M_n}\right) = \left(\log(\mathbf{e}/Z_\infty), \sum_{d \in \mathcal{P}(D_n)} \delta_{u^{(n)},\zeta_n + d}\right),
\]
in law, on the survival event $S$.
\end{corollary}

\begin{remark}
If the law of the decoration $\mathcal{D}$ is explicit, it becomes possible to compute the asymptotic probability for two particles within $O(1)$ distance from the minimal displacement $M_n$ to belong to distinct families. For example, setting $u^{1,n}$ and $u^{2,n}$ the labels of the smallest two individuals at generation $n$, we have
\begin{align*}
  \lim_{n \to \infty} \P(|u^{1,n} \wedge u^{2,n}| \geq n/2 ) = \P( d_2 \leq \zeta_2) = \E(e^{-d_2}),
\end{align*}
where $d_2$ the second smallest point of $D_1$, as $\zeta_2$ is distributed as an exponential random variable with parameter $1$.
\end{remark}

In the next sections, we derive some additional informations of the genealogy of particles close to the minimal displacement at time $n$, that can be extracted from the convergence in Theorem~\ref{thm:main}.

\section{The supercritical Gibbs measure}
\label{sec:applications}

In this section, we use Theorem~\ref{thm:main} to give a simple construction of the so-called supercritical Gibbs measures on $\bar{\calU}$, as obtained in \cite{BRV12}. More precisely, the aim is to mimic the construction of the critical measure describe in Section~\ref{subsec:critical}, but instead of using the derivative martingale $(Z_n)$ one use the supercritical additive martingale with parameter $\beta >1$.

For any $\beta \geq 0$, we denote by
\[
  \kappa(\beta) = \log \E\left( \sum_{|u|=1} e^{-\beta V(u)} \right) \in (-\infty,\infty].
\]
For all $n \in \N$ and $\beta \geq 0$ such that $\kappa(\beta) < \infty$, we denote by
\[
  W_n^{(\beta)} = \sum_{|u|=n} e^{-\beta V(u) - n \kappa(\beta)}.
\]
By the branching property, $(W_n^{(\beta)}, n \geq 0)$ is a non-negative martingale.

If we assume that $W_\infty^{(\beta)}>0$ a.s. on the survival event of the branching random walk, then one can use the same techniques as in Section~\ref{subsec:critical} to define a finite measure on $\bar{\calU}$ such that $\nu_\beta(B(u)) = e^{-\beta V(u)} W_\infty^{(\beta),u}$, which we call the Gibbs measure of the branching random walk. To justify this name, observe that
\begin{equation}
  \label{eqn:gibbsMeasure}
  \lim_{n \to \infty} \frac{\sum_{|u|=n} e^{-\beta V(u)}\delta_u}{\sum_{|u|=n} e^{-\beta V(u)}} = \frac{\nu_\beta }{W_\infty(\beta)} \text{ a.s.}
\end{equation}
for the topology of weak convergence.

However, under assumption \eqref{eqn:boundary}, it is well-known that $\lim_{n \to \infty} W^{(\beta)}_n = 0$ a.s. for all $\beta > 1$ (see e.g. \cite{Lyo97}). Nevertheless, the aim of this section is to obtain a convergence similar to \eqref{eqn:gibbsMeasure} for $\beta > 1$, thus defining the supercritical Gibbs measure on $\bar{\calU}$.

Let $\beta > 1$, as $\lim_{n \to \infty} W_n^{(\beta)}=0$, one has to choose a different renormalization in order to obtain a non-degenerate limit. We set
\[
  W_{n,\beta} = n^{3\beta/2} e^{n \kappa(\beta)} W_n^{(\beta)} = \sum_{|v|=n} e^{\beta (m_n - V(v))}.
\]
Madaule \cite[Theorem 2.3]{Mad15} proved there exists a random variable $W_{\beta,\infty}$ defined on the same probability space as $Z_\infty$ such that
\[\lim_{n \to \infty} (W_{n,\beta},Z_n) = (W_{\infty, \beta},Z_\infty) \text{ in law,}\]
with $W_{\infty,\beta}$ and $Z_\infty$ being a.s. either both positive or both null. For all $u \in \T$, we set
\[
  W_{n,\beta}^u = \sum_{|v|=n, v > u} e^{\beta (m_n + V(u) - V(v))},
\]
We construct a measure which gives mass $W_{\infty,\beta}^u$ to the ball $B(u)$ for all $u \in \T$. This measure is then used to study the so-called overlap of the branching random walk.

We recall that $(u^{(n)})$ are i.i.d. random elements of $\bar{\calU}$ sampled with law~$\bar{\nu}$. We denote by $(\xi^\beta_n, n \geq 1)$ the atoms of a Poisson point process with intensity $c_\beta e^x dx$, where we write~$c_\beta = c_* \E\left( \sum_{d \in D} e^{-\beta d} \right)$. We introduce the random measures on $\bar{\calU}$ defined by
\begin{equation}
  \label{eqn:supercriticalMeasure}
  \nu_{\beta,n} = \sum_{|u|=n} e^{\beta (m_n - V(u))} \delta_u \quad \text{and} \quad \nu_{\beta,\infty} = \sum_{n \in \N} Z_\infty^\beta e^{-\beta \xi^\beta_n} \delta_{u^{(n)}}.
\end{equation}
\begin{theorem}
\label{thm:cvSupercritical}
Assuming \eqref{eqn:supercritical}, \eqref{eqn:boundary}, \eqref{eqn:variance} and \eqref{eqn:integrability}, for any $\beta > 1$, we have
\[
  \lim_{n \to \infty} \nu_{\beta,n} = \nu_{\beta,\infty}\quad \text{ in law,}
\]
for the topology of the weak convergence of measures.
\end{theorem}

Note that this convergence is similar to the one observed for the critical measure in Remark~\ref{rem:cvCritical}. However, the convergence holds in distribution, and not almost surely. We also have $\nu_{\beta,\infty}(B(u)) = W_{\infty, \beta}^u$ in law, for all $u \in \T$.

\begin{proof}
By \cite[Theorem 2.3]{Mad15}, $\nu_{\beta,n}(B(u))$ converges in law for any $u \in \T$ as $n \to\infty$. Consequently, using Theorem~\ref{thm:main}, for any $u \in \T$, we can identify the law of the limit as
\begin{align*}
   \lim_{n \to \infty} \nu_{\beta,n}(B(u)) &= \sum_{k=1}^{\infty} \ind{u^{(k)}>u} \sum_{d \in \mathcal{P}(D_k)} Z_\infty^\beta e^{-\beta (\xi_k + d)}\quad \text{ in law,}\\
  &= Z_\infty^\beta \sum_{k=1}^{\infty} \ind{u^{(k)}>u} e^{-\beta \xi_k}\sum_{d \in \mathcal{P}(D_k)}  e^{-\beta d}.
\end{align*}
Setting $X^\beta_k = -\frac{1}{\beta} \log \sum_{d \in \mathcal{P}(D_k)} e^{-\beta d}$, we have
\[
  \lim_{n \to \infty} \nu_{\beta,n}(B(u)) = \sum_{k=1}^{\infty} \ind{u^{(k)}>u} Z_\infty^\beta e^{-\beta (\xi_k + X^\beta_k)} \quad \text{ in law}.
\]
Moreover, as $(\xi_k+X^\beta_k, k \in \N)$ are the atoms of a Poisson point process with intensity $c_\beta  e^{x} dx$ independent of $(u^{(k)})$, we conclude that for any $j \in \N$,
\[
  \lim_{n \to \infty} \left(\nu_{\beta,n}(B(u)), u \in \T(j) \right) = \left( \nu_{\beta,\infty}(B(u)), u \in \T(j)\right) \quad \text{ in law},
\]
which concludes the proof.
\end{proof}

\begin{remark}
A straightforward consequence of this result is that under assumptions \eqref{eqn:supercritical}, \eqref{eqn:boundary}, \eqref{eqn:variance} and \eqref{eqn:integrability}, the solution $Y_\beta$ of the supercritical smoothing transform (see e.g. \cite{ABM})
\[
  Y_\beta \egaldistr \sum_{|u|=1} e^{-\beta V(u)} Y^{(u)}_\beta
\]
can be written $Y_\beta = Z_\infty^\beta\sum_{k=1}^{\infty} e^{-\beta \xi^\beta_k}$.
\end{remark}

Theorem~\ref{thm:cvSupercritical} indirectly implies a proof of the conjecture of Derrida and Spohn \cite{DeS88}: the rescaled distribution of the genealogy of the most recent common ancestor of two particles chosen independently at random according to the measure $\nu_{n,\beta}/\nu_{n,\beta}(\bar{\calU})$ converges in law toward a random measure on $[0,1]$ with no mass on $(0,1)$. 
\begin{theorem}
\label{thm:DeS88}
For any $n \in \N$ and $\beta > 1$, we set
\begin{equation}
  \omega_{n,\beta} = W_{n,\beta}^{-2}\sum_{|u|=|v|=n} e^{\beta (2 m_n -  V(u)-V(v))} \delta_{|u\wedge v|/n}.
\end{equation}
Assuming \eqref{eqn:supercritical}, \eqref{eqn:boundary}, \eqref{eqn:variance} and \eqref{eqn:integrability}, conditionally on $S$, we have 
\[
  \lim_{n \to \infty} \omega_{n,\beta} = (1-\pi_\beta) \delta_0 + \pi_\beta \delta_1 \quad \text{in law,}
\]
where $\pi_\beta = \sum_{k = 1}^{\infty} p^2_k$ and $(p_k, k \geq 1)$ is a Poisson-Dirichlet mass partition with parameters $(\beta^{-1},0)$.
\end{theorem}

A similar result was already known for multiple types of Gaussian processes with a logarithmic correlation structure, such as the Generalized Random Energy Model \cite{BoK}, log-correlated Gaussian fields such as the Gaussian Free Field \cite{ArZ}, or the binary branching random walk with Gaussian increments \cite{Jag}. More precisely, it is proved that a measure similar to $\nu_{n,\beta}/\nu_{n,\beta}(\bar{\mathcal{U}})$ converges in law toward a Ruelle probability cascade. Ouimet \cite{Oui} recently extended this family of results to Gaussian fields with scale-dependent variance. Theorems~\ref{thm:cvSupercritical} and~\ref{thm:DeS88} represent extensions of these results to branching random walks with non-Gaussian increments. Contrarily to what was done in this past literature, the proof relies on the study of the extremal point process instead of proving Ghirlanda-Guerra type identities. Thus Poisson-Dirichlet distributions appear as simple functional of a Poisson point process instead of the application of Talagrand's identity (see \cite[Remark 3.8]{Jag}).

\begin{proof}
We first observe that it is enough to prove that conditionally on $S$,
\begin{equation}
  \label{eqn:target} \forall t \in (0,1) \quad \lim_{n \to \infty} \omega_{n,\beta}((t,1]) = \pi_\beta\quad \text{in law}.
\end{equation}
Indeed, the function $t \mapsto \omega_{n,\beta}((t,1])$ is decreasing on $[0,1]$, therefore \eqref{eqn:target} and Slutsky's lemma imply the convergence of the finite-dimensional distributions of the tail of $\omega_{n,\beta}$, which concludes the proof.

For $k \leq n $ and $t \in [0,1]$, we set
\begin{align*}
  \Lambda_n^k &= W_{n,\beta}^{-2}\sum_{|u|=|v|=n} e^{2\beta m_n-\beta V(u)-\beta V(v)} \ind{|u\wedge v| \geq k}\\
  \text{and}\quad \Delta_n^{k,t} &= W_{n,\beta}^{-2}\sum_{|u|=|v|=n} e^{2\beta m_n-\beta V(u)-\beta V(v)} \ind{|u\wedge v| \in [k,tn]}.
\end{align*}
We observe that for every $k \in [1,tn) \cap \N$, we have
\begin{equation}
  \label{eqn:gendarme}
  \Lambda_n^k -\Delta_n^{k,t} \leq \omega_{n,\beta}((t,1]) \leq \Lambda_n^k  \quad \text{a.s. on S.}
\end{equation}

By Theorem~\ref{thm:main}, as $S = \{ Z_\infty > 0\}$ a.s. we have
\[
  \lim_{n \to \infty} \Lambda_n^k = \Lambda_\infty^k := \frac{ \sum_{|u|=k} \left(\sum_{j=1}^{\infty} \ind{u^{(j)}>u} e^{-\beta \xi^\beta_j} \right)^2 }{\left( \sum_{j=1}^{\infty}  e^{-\beta \xi^\beta_j} \right)^2} \quad \text{in law on $S$.}
\]
Moreover, as $\nu$ is non-atomic by Lemma~\ref{lem:nonAtomic}, letting $k \to \infty$ we have
\[ \lim_{k \to \infty} \Lambda_\infty^k = \frac{\sum_{j=1}^{\infty} e^{-2\beta \xi^\beta_j}}{\left( \sum_{j=1}^{\infty} e^{-\beta \xi^\beta_j} \right)^2} \quad \text{ a.s.}\]
Using \cite[Proposition 10]{PiY97}, we have $\lim_{k \to \infty} \Lambda_\infty^k = \pi_\beta$ a.s.

We now study the asymptotic behaviour of $\Delta_n^{k,t}$, more precisely we prove that for any $\delta > 0$,
\begin{equation}
  \label{eqn:deltato0}
  \lim_{k \to \infty} \limsup_{n \to \infty} \P\left( \Delta_n^{k,t} > \delta, S \right) = 0.
\end{equation}
By \cite[Theorem 2.3]{Mad15}, $\displaystyle  \lim_{\epsilon \to 0}  \lim_{n \to \infty} \P\left( n^{3\beta/2} W_{n,\beta} \leq \epsilon, S \right) = 0$, therefore it is enough to prove that for any $\epsilon > 0$,
\begin{equation}
  \label{eqn:epsilonto0}
  \lim_{k \to \infty} \limsup_{n \to \infty} \P\left( \sum_{|u|=|v|=n} \ind{|u \wedge v| \in [k,tn]} e^{\beta(m_n-V(u)) + \beta(m_n-V(v))} > \delta \epsilon^2 \right)= 0.
\end{equation}
The proof of this result, rather technical, is postponed to Lemma~\ref{lem:entangled}.

Let $x \in [0,1]$ and $\delta > 0$, using \eqref{eqn:gendarme}, we have
\[
  \P(\Lambda_n^k \leq x + \delta, S) + \P(\Delta_n^{k,t} \geq \delta,S) \geq \P\left( \omega_{n,\beta}((t,1]) \leq x,S \right) \geq \P\left( \Lambda_n^k \leq x,S \right).
\]
Thus, letting $n$ then $k$ grows to $\infty$ and using \eqref{eqn:deltato0}, for any $t \in (0,1)$, $\omega_{n,\beta}((t,1])$ converges in law toward $\pi_\beta$ on $S$, proving \eqref{eqn:target}.
\end{proof}

Note that this proof can easily be adapted to the convergence of the overlap measure of more than two particles.

\begin{remark}
With similar computations, we can obtain a ``local limit'' convergence for the genealogy of two particles sampled according to the Gibbs measure. In effect, if we consider the non-rescaled measure
\[
  \lambda_{n,\beta} = W_{n,\beta}^{-2} \sum_{|u|=|v|=n} e^{\beta (2 m_n - V(u)-V(v))} \delta_{|u \wedge v|},
\]
we obtain $\lim_{n \to \infty} \lambda_{n,\beta} = \lambda_{\infty,\beta}$ in law on $S$, where $(p_k)$ is a Poisson-Dirichlet distribution with parameters $(\beta^{-1},0)$ and $\lambda_{\infty,\beta} = \sum_{k,k'=1}^{\infty} p_k p_{k'} \delta_{|u^{(k)}\wedge u^{(k')}|}$. Note that $\lambda_{\infty,\beta}(\{\infty\}) = \pi_\beta$.
\end{remark}

\begin{remark}
\label{rem:gibbssub(cri)tical}
Chauvin and Rouault \cite{CR} studied similarly the overlap of subcritical measures, such that $\beta < 1$. They proved that in this case, the measure $\omega_{n,\beta}$ converges toward $\delta_0$, and the measure $\lambda_{n,\beta}$ converges toward a proper probability measure on $\N$. For the critical case, Pain \cite{Pai17+} proves that if $(\beta_n)$ is a sequence converging to $1$, then $\lim_{n \to \infty} \omega_{n,\beta_n} = \delta_0$ in probability.
\end{remark}

\section{The decoration as the close relatives of minimal displacement}
\label{sec:decoration}

In this section, we prove that the law $\calD$ is the limiting distribution of the relative positions of the family of the particle that realizes the minimal displacement at time $n$. This result is similar to the one obtained in \cite{ABBS13} for branching Brownian motion. For any $n \in \N$, we denote by $\hat{u}_n$ a particle alive at time $n$ such that $V(u) = M_n$, for example the one which is the smallest for the lexicographical order on $\calU$.
\begin{theorem}
\label{thm:abbs}
For any $n \in \N$ and $k < n$, we set
\begin{equation}
  \rho_{n,k} = \sum_{|u|=n} \ind{|u \wedge \hat{u}_n| \geq k} \delta_{V(u)-M_n}.
\end{equation}
Under the assumptions of Theorem~\ref{thm:main}, we have 
\[\lim_{k \to \infty} \lim_{n \to \infty} \rho_{n,k} = \lim_{k \to \infty} \tilde{\lim_{n \to \infty}} \rho_{n,n-k} = D_1 \quad \text{in law,}\]
where $\displaystyle \tilde{\lim_{n \to \infty}} \rho_{n,n-k}$ represents any accumulation point for the sequence $(\rho_{n,n-k})$ as $n \to \infty$.
\end{theorem}

Observe that by Corollary~\ref{cor:main}, the triangular array $(\rho_{n,k}, n \geq 1, k \leq n)$ is tight. Indeed, for any continuous positive function $f$, we have
\[
  \rho_{n,k}(f) \leq \rho_{n,0}(f) = \sum_{|u|=n} f(V(u)-M_n).
\]
A straightforward consequence of Theorem~\ref{thm:abbs} is the following, more intuitive convergence.
\begin{corollary}
\label{cor:abbs}
Let $(k_n)$ be such that $\lim_{n \to \infty} k_n = \lim_{n \to \infty} n - k_n = \infty$. Under the assumptions of Theorem~\ref{thm:main},
\begin{equation}
  \lim_{n \to \infty} \rho_{n,k_n} = D_1 \quad \text{ in law}.
\end{equation}
\end{corollary}

\begin{proof}
We observe that for any $i \leq j \leq k$, and any continuous positive function $f$, we have $\rho_{n,i}(f) \geq \rho_{n,j}(f) \geq \rho_{n,k}(f)$. Consequently, for any $k \in \N$, for all $n \geq 1$ large enough, we have $\rho_{n,k}(f) \geq \rho_{n,k_n}(f) \geq \rho_{n,n-k}(f)$. Applying Theorem~\ref{thm:abbs}, we have
\[
  \lim_{k \to \infty} \limsup_{n \to \infty} \P(\rho_{n,k_n}(f)-\rho_{n,k}(f) > \epsilon) = 0 \quad \text{ for any } \epsilon > 0,
\]
which concludes the proof.
\end{proof}

The first limit in distribution for Theorem~\ref{thm:abbs} is a straightforward consequence of Fact~\ref{fct:Madaule}.
\begin{lemma}
\label{lem:oneSideabbs}
We have 
\begin{equation}
  \displaystyle \lim_{k \to \infty} \lim_{n \to \infty} \rho_{n,k} = D \quad \text{ in law.}
\end{equation}
\end{lemma}

\begin{proof}
Using Fact~\ref{fct:Madaule}, we observe that for any $k \in \N$, conditionally on $\calF_k$,
\[
  \lim_{n \to \infty} \left( \sum_{|v|=n,v > u} \delta_{V(v)-m_n}, Z_n^u, u \in \T(k)\right) = \left( \mu^u_\infty, Z^u_\infty, u \in \T(k) \right) \text{ in law on $S$.}
\]
Therefore, setting $M_n^u = \min_{|v|=n, v > u} V(v)$, we have in particular
\[
  \lim_{n \to \infty} \sum_{|u|=k} \ind{M_n^u = M_n} \sum_{|v|=n,v>u} \delta_{V(v)-m_n} = \sum_{|u|=k} \ind{u^{(1)}>u} \mu^u_\infty \quad \text{ in law on $S$.}
\]
Observe that $\sum_{|u|=k} \ind{u^{(1)}>u} \mu^u_\infty = \sum_{n=1}^{\infty}\ind{u^{(n)}_k = u^{(1)}_k} \theta_{\xi_n -\log Z_\infty} D_n$.

Let $f$ be a continuous positive function with compact support, we prove that
\begin{equation}
  \label{eqn:convergenceSupplementaire}
  \lim_{k \to \infty} \sum_{n=2}^{\infty} \ind{u^{(n)}_k = u^{(1)}_k} \sum_{d \in \mathcal{P}(D_n)}  f(\xi_n + d-\log Z_\infty) = 0 \quad \text{ in probability}.
\end{equation}
In effect, for any $k \in \N$, we have
\begin{multline*}
  \E\left( \sum_{n=2}^{\infty} \sum_{d \in \mathcal{P}(D_n)} f(\xi_n+d-\log Z_\infty) \ind{u^{(n)}_k = u^{(1)}_k} \middle| \calF_k \right)\\
   = \sum_{|u|=k} \bar{\nu}(B(u)) \sum_{n=2}^{\infty} \E\left(\ind{u^{(n)}_k = u} g(\xi_n - \log Z_\infty)\middle|\calF_k\right),
\end{multline*}
where $g : x \mapsto \E\left( \sum_{d \in \mathcal{P}(D)} f(x + d) \right)$. Therefore,
\begin{multline*}
  \E\left( \sum_{n=2}^{\infty} \sum_{d \in \mathcal{P}(D_n)} f(\xi_n+d-\log Z_\infty) \ind{u^{(n)}_k = u^{(1)}_k} \middle| \calF_k \right)\\ = \left( \sum_{|u|=k} \bar{\nu}(B(u))^2 \right) \sum_{n=2}^{\infty} \E(g(\xi_n-\log Z_\infty)|\calF_k).
\end{multline*}
As $\lim_{k \to \infty} \max_{|u|=k} \bar{\nu}(B(u)) = 0$ a.s. (see Lemma~\ref{lem:nonAtomic}), we conclude that \eqref{eqn:convergenceSupplementaire} holds. This result yields that $\lim_{k \to \infty} \sum_{|u|=k} \ind{u^{(1)}>u} \mu^u_\infty(f) = \theta_{\xi_1-\log Z_\infty} D(f)$ in law. We conclude the proof observing that we chose the law of the decoration such that $\min D = 0$ a.s.
\end{proof}

To complete the proof of Theorem~\ref{thm:abbs}, we first observe that the genealogy of particles close to the minimal displacement at time $n$ in the branching random walk are either close relatives, or their most recent common ancestor is a close relative to the root. This well-known estimate can be found for example in \cite[Theorem 4.5]{Mal16+}. For any $z \geq 1$, we have
\begin{equation}
  \label{eqn:genealogy}
    \lim_{k \to \infty} \limsup_{n \to \infty} \P\left( \exists u,v \in \T : \begin{array}{l}|u|=|v|=n, V(u),V(v) \leq m_n + z,\\ |u\wedge v| \in [k,n-k]\end{array} \right) = 0.
\end{equation}

\begin{lemma}
\label{lem:otherSideabbs}
For any $k \in \N$, we set $(n^k_p, p \geq 1)$ an increasing sequence such that $(\rho_{n^k_p,n^k_p-k})$ converges. We have $\lim_{k \to \infty}  \lim_{p \to \infty} \rho_{n^k_p,n^k_p-k} = D$ in law.
\end{lemma}

\begin{proof}
For any positive continuous function $f$ with compact support and $k \in \N$, we have
\[
  \rho_{n,k}(f) - \rho_{n,n-k}(f) = \sum_{|u|=n} \ind{|u\wedge \hat{u}_n| \in [k,n-k]} f(V(u)-M_n).
\]
We write $z = \sup\{x \geq 0 : f(x) > 0 \}$, for any $y \geq 0$, we have
\begin{multline*}
  \P\left( \rho_{n,k}(f) - \rho_{n,n-k}(f)>0 \right)
  \leq \P\left( \exists u \in \T(n) : \begin{array}{l}|u\wedge \hat{u}_n| \in [k,n-k],\\ V(u)-M_n \leq z\end{array} \right)\\
  \leq \P(M_n - m_n \geq y) + \P\left( \exists u,v \in \T(n) : \begin{array}{l}|u\wedge v| \in [k,n-k],\\ V(u),V(v) \leq m_n + y + z \end{array} \right).
\end{multline*}
Letting $n$ then $k \to \infty$, we have by \eqref{eqn:genealogy},
\begin{equation*}
  \limsup_{k \to \infty} \limsup_{n \to \infty} \P\left( \rho_{n,k}(f) - \rho_{n,n-k}(f)>0 \right)\leq \sup_{n \in \N} \P(M_n \geq m_n + y).
\end{equation*}
Moreover, $(M_n-m_n)$ is tight, by \cite{Aid13}, thus letting $y \to \infty$, we conclude that
\[
  \lim_{k \to \infty} \limsup_{n \to \infty} \P\left( \rho_{n,k}(f) - \rho_{n,n-k}(f)>0 \right) = 0.
\]
Using Lemma~\ref{lem:oneSideabbs}, we conclude the proof.
\end{proof}

We were not able to study the limiting distribution of $\rho_{n,n-k}$, but this law can probably be constructed, similarly to the point process $\lim_{t \to \infty} \mathscr{Q}(t,\zeta)$ defined in \cite{ABBS13} for the branching Brownian motion.
\begin{conjecture}
For any $k \in \N$, there exists a point measure $\rho_k$ such that $\lim_{n \to \infty} \rho_{n,n-k} = \rho_k$.
\end{conjecture}

\appendix

\section{Some technical results}

In this section, we provide some technical estimates on the branching random walks. We first prove that \eqref{eqn:integrability} is equivalent to the usual integrability conditions for the branching random walk.
\begin{lemma}
\label{lem:equivalent}
Under assumptions \eqref{eqn:supercritical}, \eqref{eqn:boundary} and \eqref{eqn:variance}, the condition \eqref{eqn:integrability} is equivalent to
\begin{align*}
  &\E\left( \sum_{|u|=n} e^{-V(u)} \log_+ \left(\sum_{|u|=n} e^{-V(u)} \right)^2 \right)  < \infty\\
    &\E\left( \sum_{|u|=n} V(u)_+ e^{-V(u)} \log_+ \left(\sum_{|u|=n} V(u)_+ e^{-V(u)} \right) \right) < \infty
\end{align*}
\end{lemma}

\begin{proof}
The reciprocal part is a direct consequence of \cite[Lemma B.1]{Aid13}. To prove the direct part, we first observe that by \eqref{eqn:integrability},
\begin{multline*}
  \qquad \qquad \E\left( \sum_{|u|=n} e^{-V(u)} \log_+ \left(\sum_{|u|=n} e^{-V(u)} \right)^2 \right)\\
  \leq \E\left( \sum_{|u|=n} e^{-V(u)} \log_+ \left(\sum_{|u|=n} (1 + V(u)_+)e^{-V(u)} \right)^2 \right)< \infty.\qquad \qquad 
\end{multline*}

We now use the celebrated spinal decomposition of the branching random walk, introduced by Lyons \cite{Lyo97}. Loosely speaking, it is an alternative description of the law of the branching random walk biased by the martingale $(W_n)$, as the law of a branching random walk $(\T,V)$ with a distinguished spine $w \in \partial \T$ that makes more children than usual. For any $u \in \T$, we write $\xi(u) = \log_+ \left(\sum_{v \in \Omega(\pi u)} \sum_{|u|=1} V(u)_+ e^{-V(u)}\right)$. We denote by $\hat{\P}=W_n.\P$ the size-biased distribution, and refer to \cite{Lyo97} for more details on the spinal decomposition. We have
\begin{align*}
 &\E\left( \sum_{|u|=n} V(u)_+ e^{-V(u)} \log_+ \left(\sum_{|u|=n} V(u)_+ e^{-V(u)} \right) \right)\\
 = &\hat{\E}\left( \xi(w_1) V(w_1)_+ \right)\\
 \leq &\hat{\E}\left( V(w_1)^2 \right)^{1/2} \hat{\E}\left( \xi(w_1)^2 \right)^{1/2} < \infty,
\end{align*}
by Cauchy-Schwarz inequality, using \eqref{eqn:variance} and \eqref{eqn:integrability} to conclude.
\end{proof}

We now prove that \eqref{eqn:epsilonto0} holds.
\begin{lemma}
\label{lem:entangled}
For any $\beta > 1$ and $k \leq n$, we set
\[
  R_{n,k}^\beta = \sum_{|u|=|v|=n} \ind{|u \wedge v| \in [k,n-k]} e^{\beta(m_n-V(u)) + \beta(m_n-V(v))}.
\]
For any $\epsilon>0$, we have $\displaystyle \lim_{k \to \infty} \limsup_{n \to \infty} \P\left(  R_{n,k}^\beta \geq \epsilon \right) = 0$.
\end{lemma}

\begin{proof}
To prove this result, we first introduce some notation. For any $u \in \T$, we set
\[
  \xi(u) = \log \sum_{|v|=|u|+1,v>u}  (1 + (V(v)-V(u))_+ ) e^{V(u)-V(v)}.
\]
For any $n \in \N$ and $k \leq n$, we write $f_n(k) = \frac{3}{2} \log \frac{n+1}{n-k+1}$ and, for $y,z,h \geq 0$,
\begin{align*}
  \mathcal{A}_n(y)& = \left\{ |u| \leq n : V(u_j) \geq f_n(j) - y, j \leq |u| \right\},\\
  \bar{\mathcal{A}}_n(y,h) &= \left\{ |u| = n : u \in \mathcal{A}_n(y), V(u) - f_n(n) + y \in [h-1,h] \right\}\\
  \mathcal{B}_n(y,z)& = \left\{ |u| \leq n : \xi(u_j) \leq z + (V(u_j) - f_n(j) + y)/2 , j \leq |u| \right\}.
\end{align*}
We introduce branching random walk estimates obtained in \cite{Mal16+}. There exist $C>0$ and a function $\chi$ such that $\lim_{z \to \infty} \chi(z)=0$ such that for any $k \leq n$ and $y,z,h \geq 1$ we have
\begin{align}
  &\P\left( \exists u, v \in \bar{\mathcal{A}}_n(y,h) \cap \mathcal{B}_n(y,z) : |u \wedge v| \in [k,n-k] \right) \leq C \frac{zyh^2e^{2h-y}}{k^{1/2}},\nonumber\\
  &\P\left( \bar{\mathcal{A}}_n(y,h) \cap \mathcal{B}^c_n(y,z) \neq \emptyset \right) \leq  \chi(z)yhe^{h-y},\nonumber \\
    &\P\left( \mathcal{A}_n(y) \neq \emptyset \right) \leq C y e^{-y} \quad \text{ and } \quad \E\left( \# \bar{\mathcal{A}}_n(y,h) \right) \leq C yhe^{h-y}. \label{eqn:listestimates}
\end{align}
In the rest of this proof, $C$ is a large positive constant, that depends only on the law of the branching random walk, and may change from line to line.

We decompose $R_{n,k}^\beta$ into three parts, that we bound separately. For any $h \geq 0$, we have
\[
  R_{n,k}^\beta \leq \tilde{R}^\beta_{n,k}(h) + 2 W_{n,\beta}(h) W_{n,\beta},
\]
where we write
\[\tilde{R}^\beta_{n,k}(y,h) = \sum_{|u|=|v|=n} \ind{|u \wedge v| \in [k,n-k]} \ind{V(u)-m_n \leq h} e^{\beta(m_n-V(u)) + \beta(m_n-V(v))},\]
and $W_{n,\beta}(h) = \sum_{|u|=n} \ind{V(u)-m_n \geq h} e^{\beta(m_n-V(u))}$.

By \eqref{eqn:listestimates}, for any $y,h \geq 0$, we have
\begin{align*}
  &\E\left( \sum_{|u|=n} \ind{u \in \mathcal{A}_n(y),V(u) \geq m_n + h} e^{\beta (m_n - V(u))} \right) \\
   = &\sum_{j=h+1}^{\infty} e^{-\beta(j-1)}\E\left( \# \bar{\mathcal{A}}_n(y,j+y) \right)\\
  \leq &C y e^{-y} \sum_{j=h+1}^{\infty} (j+y)e^{(1-\beta)j} \leq C y (h+y) e^{(1-\beta)h}.
\end{align*}
We have $E\left( \sum_{|u|=n} \ind{u \in \mathcal{A}_n(y)} e^{\beta (m_n - V(u))} \right) \leq C y e^{(\beta-1)y}$ by similar computations. Using the Markov inequality, there exists $C>0$ such that for any $\epsilon \geq 0$ and $y,h \geq 0$, we have
\begin{align*}
  \P\left( W_{n,\beta}(h) \geq \epsilon \right) &\leq \P\left( \mathcal{A}_n(y) \neq \emptyset \right) + \frac{C}{\epsilon} y (h+y) e^{(1-\beta)h}\\
  &\leq C y e^{-y}+ \frac{C}{\epsilon} y(h+y)e^{(1-\beta)h},
\end{align*}
and similarly for any $A>0$, $\P(W_{n,\beta} \geq A) \leq C ye^{-y} + C y^2 e^{(\beta-1)y}/A$. Thus, for any $\delta > 0$, we have
\begin{align*}
  \P\left( W_{n,\beta}(h) W_{n,\beta} \geq \delta \right)
  &\leq \P(W_{n,\beta}(h) \geq \delta\epsilon) + \P(W_{n,\beta} \geq 1/\epsilon)\\
  &\leq C y e^{-y} + C y (h+y) e^{(1-\beta)h}/(\delta \epsilon) + C \epsilon y e^{(\beta-1)y}.
\end{align*}
Choosing $y \geq 1$ large enough, then $\epsilon>0$ small enough and $h$ large enough, we obtain
\[
  \sup_{n \in \N} \P\left( W_{n,\beta}(h) W_{n,\beta} \geq 2\delta \right) \leq \delta.
\]

In a second time,we bound $\tilde{R}^\beta_{n,k}$, by observing that for any $y,z \geq 0$,
\begin{align*}
  \P(\tilde{R}^\beta_{n,k}(h) \neq 0) &\leq \P(\mathcal{A}_n(y) \neq \emptyset) + \sum_{j=0}^{h+y}\P\left( \bar{\mathcal{A}}_n(y,j) \cap \mathcal{B}^c_n(y,z) \neq \emptyset \right)\\
   &\qquad   + \sum_{j=0}^{h+y}\P\left( \exists u, v \in \bar{\mathcal{A}}_n(y,j) \cap \mathcal{B}_n(y,z) : |u \wedge v| \in [k,n-k] \right)\\
   \leq & C y e^{-y} + \chi(z) y (h+y)e^{h} + C \frac{zy(y+h)^2e^{2h+y}}{k^{1/2}},
\end{align*}
using again \eqref{eqn:listestimates}.

As a consequence, for any $\delta > 0$, we can choose $y \geq 1$, $\epsilon>0$, and $h \geq 0$ large enough such that for any $k,z \geq 0$ and $n \geq k$, we have
\[
  \P\left( R_{n,k}^\beta \geq \delta \right) \leq \delta + \chi(z)y (h+y)e^{h} + C \frac{zy(y+h)^2e^{2h+y}}{k^{1/2}}.
\]
Setting $z = k^{1/4}$, we obtain
$
  \limsup_{k \to \infty} \limsup_{n \to \infty} \P\left( R_{n,k}^\beta \geq \delta \right) \leq \delta,
$
which concludes the proof.
\end{proof}

\paragraph*{Acknowledgements.} I wish to thank Thomas Madaule and Julien Barral for many useful discussions, as well as pointing me references \cite{SuZ15} and \cite{BKNSW} respectively. 

\newcommand{\etalchar}[1]{$^{#1}$}


\begin{thebibliography}{BKN{\etalchar{+}}14}

\bibitem[A{\"{\i}}d13]{Aid13}
E.~A{\"{\i}}d{\'e}kon.
\newblock Convergence in law of the minimum of a branching random walk.
\newblock {\em Ann. Probab.}, 41(3A):1362--1426, 2013.

\bibitem[ABBS13]{ABBS13}
E.~A{\"{\i}}d{\'e}kon, J.~Berestycki, {\'E}.~Brunet, and Z.~Shi.
\newblock Branching {B}rownian motion seen from its tip.
\newblock {\em Probab. Theory Related Fields}, 157(1-2):405--451, 2013.

\bibitem[AS14]{AiS14}
E.~A{\"{\i}}d{\'e}kon and Z.~Shi.
\newblock The {S}eneta-{H}eyde scaling for the branching random walk.
\newblock {\em Ann. Probab.}, 42(3):959--993, 2014.

\bibitem[ABR09]{ABR09}
L.~Addario-Berry and B.~A. Reed.
\newblock Minima in branching random walks.
\newblock {\em Ann. Probab.}, 37(3):1044--1079, 2009.

\bibitem[ABM12]{ABM}
G. Alsmeyer, J.~D. Biggins, and M. Meiners.
\newblock The functional equation of the smoothing transform.
\newblock {\em Ann. Probab.}, 40(5):2069--2105, 2012.

\bibitem[ABK12]{ABK12}
L.-P. Arguin, A.~Bovier, and N.~Kistler.
\newblock Poissonian statistics in the extremal process of branching Brownian motion.
\newblock {\em Ann. Appl. Probab.}, 22(4):1693--1711, 2012.

\bibitem[ABK13]{ABK13}
L.-P. Arguin, A.~Bovier, and N.~Kistler.
\newblock The extremal process of branching {B}rownian motion.
\newblock {\em Probab. Theory Related Fields}, 157(3-4):535--574, 2013.

\bibitem[AZ14]{ArZ}
L.-P. Arguin, O. Zindy.
\newblock Poisson-{D}irichlet statistics for the extremes of a log-correlated {G}aussian field
\newblock {\em Ann. Appl. Probab.}, 24(4):1446--1481, 2014.

\bibitem[BKN{\etalchar{+}}14]{BKNSW}
J.~Barral, A.~Kupiainen, M.~Nikula, E.~Saksman, and C.~Webb.
\newblock Critical {M}andelbrot cascades.
\newblock {\em Comm. Math. Phys.}, 325(2):685--711, 2014.

\bibitem[BRV12]{BRV12}
J. Barral, R. Rhodes, and V. Vargas.
\newblock Limiting laws of supercritical branching random walks.
\newblock {\em C. R. Math. Acad. Sci. Paris}, 350(9-10):535--538, 2012.

\bibitem[BG11]{BeG11}
J.~B{\'e}rard and J.-B. Gou{\'e}r{\'e}.
\newblock Survival probability of the branching random walk killed below a
  linear boundary.
\newblock {\em Electron. J. Probab.}, 16:no. 14, 396--418, 2011.

\bibitem[Big76]{Big76}
J.~D. Biggins.
\newblock The first- and last-birth problems for a multitype age-dependent
  branching process.
\newblock {\em Advances in Appl. Probability}, 8(3):446--459, 1976.

\bibitem[BK04]{BiK04}
J.~D. Biggins and A.~E. Kyprianou.
\newblock Measure change in multitype branching.
\newblock {\em Adv. in Appl. Probab.}, 36(2):544--581, 2004.

\bibitem[BL18]{BiL}
M.~Biskup and O.~Louidor.
\newblock Full extremal process, cluster law and freezing for two-dimensional
  discrete {G}aussian free field.
\newblock {\em Adv. Math.}, 330:589--687, 2018.

\bibitem[BH17]{BoH}
A.~Bovier and L.~Hartung.
\newblock Extended convergence of the extremal process of branching {B}rownian
  motion.
\newblock {\em Ann. Appl. Probab.}, 27(3):1756--1777, 2017.

\bibitem[BoK04]{BoK}
A. Bovier and I. Kurkova.
\newblock Derrida's generalized random energy models 2: Models with continuous hierarchies.
\newblock {\em Ann. Inst. H. Poincar\'e Probab. Statist.}, 40(4):481--495, 2004.

\bibitem[Bur09]{Bur09}
D.~Buraczewski.
\newblock On tails of fixed points of the smoothing transform in the boundary
  case.
\newblock {\em Stochastic Process. Appl.}, 119(11):3955--3961, 2009.

\bibitem[BDK18]{BDK16+}
D.~Buraczewski, P.~Dyszewski, and K.~Kolesko.
\newblock Local fluctuations of critical Mandelbrot cascades.
\newblock To appear in {\em  	Ann. Henri Poincaré}, 2018+.

\bibitem[CHL17]{CHL17}
A.~Cortines, L.~Hartung and O.~Louidor.
\newblock The Structure of Extreme Level Sets in Branching Brownian Motion.
\newblock {\em arXiv:1703.06529}, 2017.

\bibitem[CR97]{CR}
B.~Chauvin and A.~Rouault.
\newblock Boltzmann-{G}ibbs weights in the branching random walk.
\newblock In {\em Classical and modern branching processes ({M}inneapolis,
  {MN}, 1994)}, volume~84 of {\em IMA Vol. Math. Appl.}, pages 41--50.
  Springer, New York, 1997.

\bibitem[Che15]{Che15}
X.~Chen.
\newblock A necessary and sufficient condition for the nontrivial limit of the
  derivative martingale in a branching random walk.
\newblock {\em Adv. in Appl. Probab.}, 47(3):741--760, 2015.

\bibitem[CMM15]{CMM}
X.~Chen, T. Madaule and B. Mallein.
\newblock  On the trajectory of an individual chosen according to supercritical Gibbs measure in the branching random walk.
\newblock {\em arXiv:1507.04506}, 2015.

\bibitem[DS88]{DeS88}
B.~Derrida and H.~Spohn.
\newblock Polymers on disordered trees, spin glasses, and traveling waves.
\newblock {\em J. Statist. Phys.}, 51(5-6):817--840, 1988.
\newblock New directions in statistical mechanics (Santa Barbara, CA, 1987).

\bibitem[Fel71]{Fel}
William Feller.
\newblock {\em An introduction to probability theory and its applications.
  {V}ol. {II}.}
\newblock Second edition. John Wiley \& Sons, Inc., New York-London-Sydney,
  1971.

\bibitem[HS09]{HuS09}
Y.~Hu and Z.~Shi.
\newblock Minimal position and critical martingale convergence in branching
  random walks, and directed polymers on disordered trees.
\newblock {\em Ann. Probab.}, 37(2):742--789, 2009.

\bibitem[Jag16]{Jag}
A. Jagannath.
\newblock On the overlap distribution of branching random walks.
\newblock {\em Electron. J. Probab.}, 21, 2016.

\bibitem[Kal02]{Kal02}
O.~Kallenberg.
\newblock {\em Foundations of modern probability}.
\newblock Probability and its Applications (New York). Springer-Verlag, New
  York, second edition, 2002.

\bibitem[Lyo97]{Lyo97}
R.~Lyons.
\newblock A simple path to {B}iggins' martingale convergence for branching
  random walk.
\newblock In {\em Classical and modern branching processes ({M}inneapolis,
  {MN}, 1994)}, volume~84 of {\em IMA Vol. Math. Appl.}, pages 217--221.
  Springer, New York, 1997.

\bibitem[Mad17]{Mad15}
T.~Madaule.
\newblock Convergence in {L}aw for the {B}ranching {R}andom {W}alk {S}een from
  {I}ts {T}ip.
\newblock {\em J. Theor. Probab.}, 30(1):27--63, 2017.

\bibitem[Mad16]{Mad16}
T.~Madaule.
\newblock  The tail distribution of the {D}erivative martingale and the global minimum of the branching random walk.
\newblock {\em arXiv:1606.03211}, 2016.

\bibitem[Mai13]{Mai13}
P.~Maillard.
\newblock A note on stable point processes occurring in branching Brownian motion.
\newblock {\em Electron. Comm. Probab.}, 18(5):1--9, 2013.

\bibitem[Mal16]{Mal16+}
B.~Mallein.
\newblock Asymptotic of the maximal displacement in the branching random walk.
\newblock {\em Graduate J. Math.}, 1(2):92--104, 2016.

\bibitem[Oui17]{Oui}
F.~Ouimet.
\newblock  Geometry of the Gibbs measure for the discrete 2D Gaussian free field with scale dependent variance.
\newblock {\em ALEA Lat. Am. J. Probab. Math. Stat}, 14(2):851--908, 2017.


\bibitem[Pai17]{Pai17+}
M.~Pain.
\newblock The near-critical Gibbs measure of the branching random walk.
\newblock {\em arXiv:1703.09792}, 2017.

\bibitem[PY97]{PiY97}
Jim Pitman and Marc Yor.
\newblock The two-parameter {P}oisson-{D}irichlet distribution derived from a
  stable subordinator.
\newblock {\em Ann. Probab.}, 25(2):855--900, 1997.

\bibitem[SZ15]{SuZ15}
E.~Subag and O.~Zeitouni.
\newblock Freezing and decorated {P}oisson point processes.
\newblock {\em Comm. Math. Phys.}, 337(1):55--92, 2015.

\end{thebibliography}
\end{document}